\documentclass{article}
\usepackage{authblk}
\usepackage{lineno,hyperref}
\usepackage{comment}
\modulolinenumbers[5]
\usepackage{graphics} 
\usepackage{amsmath} 
\usepackage{amssymb}  
\usepackage[american]{circuitikz}
\usetikzlibrary{arrows,automata}
\usepackage{amsthm}

\newcommand{\Real}{\mathbb{R}}
\newcommand{\diag}{\textrm{diag }}
\newtheorem{problem}{Problem}

\newtheorem{remark}{Remark}
\newtheorem{assumption}{Assumption}
\newtheorem{example}{Example}

\newtheorem{proposition}{Proposition}
\newcommand{\Img}{\textrm{Im }}
\newcommand{\matdxd}[4] {\left( \begin{array}{cc} #1 & #2\\ #3 &
                                                                 #4 \end{array} \right)}
\newcommand{\vetdc}[2] {\left( \begin{array}{cc} #1 \\#2 \end{array} \right)}
\newcommand{\vetdr}[2] {\left( \begin{array}{cc} #1 & #2 \end{array} \right)}

\title{Structured identification for network reconstruction of RC-models}

\author{Gabriele Calzavara \and Luca Consolini \and Juxhino Kavaja}
\date{All authors are with Dipartimento di Ingegneria e Architettura,
        Universit\`a degli Studi di Parma, Parco Area delle Scienze,
        181/A, 43124 Parma, Italy.
      emails: \{\texttt{gabriele.calzavara, luca.consolini,
        juxhino.kavaja}\}
      \texttt{@unipr.it}}

\begin{document}
\maketitle

\begin{abstract}
Resistive-capacitive (RC) networks are used to model various processes in engineering, physics or
biology. We consider the problem of recovering the network connection structure from measured input-output data. We address this problem as a structured
identification one, that is, we assume to have a state-space model of the system
(identified with standard techniques, such as subspace methods) and
find a coordinate transformation that puts the identified system in a
form that reveals the nodes connection structure. We characterize the solution set, that is, the set of
all possible RC-networks that can be associated to the input-output data.
We present a possible solution algorithm and show some computational experiments.
\end{abstract}


\section{Introduction}

Various dynamical models of processes in engineering, physics or
biology have the following form

\begin{equation} \label{eqn_RC_intro}
  \begin{array}{ll}
G \dot x(t)= S x(t) + B u(t)\\
    y(t)=C x(t)\,,
    \end{array}
\end{equation}

where $x(t) \in \Real^n$ is the state, $u(t) \in \Real^m$ is the
input, $y(t) \in \Real^p$ is the output.
We assume that $S \in \Real^{n \times n}$ is symmetric,
  $G \in \Real^{n \times n}$ is diagonal and positive definite,
  while $B$ and $C$ have no special structure.

  For instance, model~(\ref{eqn_RC_intro}) may represent a generic RC
  (resistive and capacitive) network, in which the components of $x \in
  \Real^n$ are the node potentials, $S$ is the admittance matrix and
  the diagonal elements of $G$ are the nodes
  capacitances. 
As an example, consider the RC circuit represented in
Figure~\ref{fig_rc_intro} and assume that the output is the
potential of node $A$. The corresponding model has form~\eqref{eqn_RC_intro}:

\begin{equation}
  \label{eqn_ex_into}
  \begin{gathered}
  \left(\begin{array}{cc}
          C_1&0 \\
          0 & C_2
               \end{array} \right) \left(\begin{array}{cc}
          \dot{v}_1(t) \\
          \dot{v}_2(t)
                                         \end{array} \right) \\ =                                         \left(\begin{array}{cc}
          -R_1^{-1}-R_{12}^{-1}&R_{12}^{-1} \\
          R_{12}^{-1}&      -R_2^{-1}-R_{12}^{-1}
               \end{array} \right) \left(\begin{array}{cc}
          v_1(t) \\
          v_2(t)
                                         \end{array} \right)\,,
                                       \end{gathered}
\end{equation}
\[
y(t) = \left(\begin{array}{cc}
          1&0 
               \end{array} \right) \left(\begin{array}{cc}
          v_1(t) \\
          v_2(t)
                                         \end{array} \right)\,.
  \]

We can associate a weighted undirected graph to matrix
$S$ in~\eqref{eqn_RC_intro} by considering $S$ a weighted adjacency
matrix. Namely, we define a node for each component of vector $x$
and we define an edge between node $i$ and node $j$ if the entry of row
$i$ and column $j$ of $S$ is nonzero. The numerical value of the entry
represents the edge weight. For instance,
Figure~\ref{fig_graph_associated_to_in_ex} represents the graph associated to the
model of the circuit in Figure~\ref{fig_rc_intro}.

For simplicity, in the rest of the paper we will
omit self loops when representing the graph associated
to the $S$ matrix of a system in form~\eqref{eqn_RC_intro}, 
  \begin{figure}
    \label{fig_rc_intro}
\begin{center}
    \begin{tikzpicture}
\path (0,0) coordinate (ref_gnd);
\draw
  (ref_gnd) to[R=\(R_1\)] ++(0,2) -- ++ (2,0)  node  (cond1) {} to [C=\(C_1\)] ++(0,-2) -- ++(-2,0);
\draw (cond1) to [R=\(R_{12}\)] ++ (2,0) node
  (cond2) {} to [C=\(C_2\)]  ++(0,-2) -- ++(-2,0);
  \draw (cond2) to  ++ (2,0)  to [R=\(R_2\)] ++(0,-2) -- ++(-2,0);

  \node[above] at (cond1) {A};
\fill[color=black] (ref_gnd)++(2,0)   circle[radius=0.08];
\fill[color=black] (ref_gnd)++(4,0) circle[radius=0.08];
\fill[color=black] (ref_gnd)++(2,2)   circle[radius=0.08];
\fill[color=black] (ref_gnd)++(4,2) circle[radius=0.08];
\end{tikzpicture}
 \caption{Example of an RC-circuit.}
\end{center}
\end{figure}
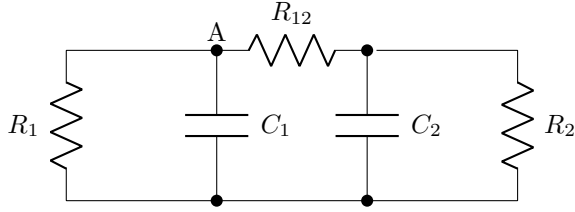

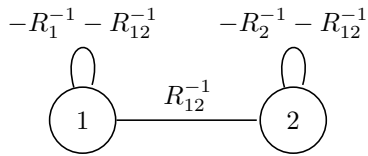
\begin{figure}
  \centering
\begin{tikzpicture}[->,>=stealth',shorten >=1pt,auto,node
  distance=2.8cm, semithick]
  \tikzstyle{every state}=[fill=white,draw=black,text=black]
\tikzset{every loop/.style={}}
  \node[state] (A)   {$1$};
  \node[state] (B) [right of=A] {$2$}; 
  \path [-]  (A) edge [loop above] node {$-R_1^{-1}-R_{12}^{-1}$} (A)
  edge node {$R_{12}^{-1}$}  (B)
  (B) edge [loop above] node {$-R_2^{-1}-R_{12}^{-1}$} (B)     ;
      \end{tikzpicture}
      \caption{Graph representation of model associated to the circuit
        in Figure~\ref{fig_rc_intro}.}
\label{fig_graph_associated_to_in_ex}
\end{figure}

  Form~(\ref{eqn_RC_intro}) can also be used to model those
  systems that have the same mathematical representation of
  RC-circuits, such as thermal systems, physical network systems
  (\cite{VANDERSCHAFT201721}), dendritic structures (\cite{ermentrout2010mathematical}), mammillary systems
  (\cite{haddad2010nonnegative}), or, more generally, diagonally symmetrizable
  compartmental systems (see~\cite{HEARON1979293}).

  Consider a second system of form
\begin{equation} \label{eqn_ident}
  \begin{array}{ll}
\dot z(t)= \hat A z(t) + \hat B u(t)\\
    y(t)=\hat C z(t)\,,
    \end{array}
\end{equation}
in which $\hat A$, $\hat B$, $\hat C$ have the same dimensions
of $S$, $B$, $C$.
In this paper we consider the following algebraic problem.

\begin{problem}
  \label{problem_id}
Consider systems~(\ref{eqn_RC_intro}),~(\ref{eqn_ident}), where
matrices $\hat A,\hat B,\hat C, B, C$ are given.
Find, if possible, an
invertible matrix $T$, a symmetric matrix $S$ and a strictly positive diagonal matrix $G$ such that
\begin{equation}
  \label{eqn_conditions}
  \left\{
  \begin{array}{ll}
T^{-1} \hat A T =  G^{-1} S \\
    \hat C T= C \\
T^{-1} \hat B =  G^{-1} B\,.
  \end{array}
  \right.
\end{equation}
\end{problem} 

In other words, we are looking for a state-space transformation of
form $z=T x$ and suitable matrices $S$ and $G$ such that~(\ref{eqn_ident}) takes on form~(\ref{eqn_RC_intro}).

Problem~\ref{problem_id} can be interpreted as a structured identification
one.  Namely, system~(\ref{eqn_ident}) represents an identified
black-box  model, obtained from experimental data with standard techniques, for instance state space
    identification methods.  Our aim is to
    check if this system, after state transformation $z=Tx$, can be given
    the form of model~(\ref{eqn_RC_intro}), and, if this is possible, find
    such a transformation. The structural requirements imposed in
    Problem~\ref{problem_id} are the following ones
    \begin{itemize}
      \item The transformed system matrix $T^{-1} \hat A T$ must
        be the product of a positive diagonal and a symmetric matrix.
      \item The input and output matrices are assigned.
      \end{itemize}
The first requirement ensures that the transformed system matrix can
be written as the product of a diagonal matrix (containing the inverse
capacities, in the case of RC-circuits) and a symmetric matrix (the admittance
matrix in the case of RC-circuits). The second requirement is due
to the fact that, in some cases, the input matrix $B$ or the output
matrix $C$ may be known from structural properties of the system at
hand. For instance, in an RC-circuit we may be able to measure
the potential of the first $n_1$ nodes, while the potentials of the
remaining ones are not accessible. In this case, $C$ would correspond
to the projection matrix on the first $n_1$ components.
Note that we may not have this requirement on $B$ or $C$. In
this case, the second, the third condition in~(\ref{eqn_conditions}),
or both could be omitted.

Further, we consider a more restrictive version of
Problem~\ref{problem_id}, based on the
observation that, in various dynamical models,
matrix $S$ in~\eqref{eqn_RC_intro} is Metzler,
that is, all its off-diagonal entries are non-negative.
For instance, in RC-networks, off-diagonal entries correspond
to the values of the resistances connecting the network nodes (see~\eqref{eqn_ex_into}). This suggests the following formulation.
\begin{problem}
  \label{problem_mer}
  Solve Problem~\ref{problem_id} with the additional requirement
  that $S$ is Metzler.
  \end{problem}
  Further, if Problem~\ref{problem_mer} has multiple solutions,
  one may minimize the number of nonzero components of $S$,
  that is, minimize $\|S\|_0$, the so-called zero-norm of $S$.
This
follows the principle of parsimony of finding the
simplest model of form~\ref{problem_id}
that fits the data. This leads to the following additional problem.
\begin{problem}
  \label{problem_mimin}
  Find the solutions of  Problem~\ref{problem_mer} in which $\|S\|_0$
  is minimum.
\end{problem}
For instance, in an RC-circuit, the nonzero elements of $S$ represent
the resistive connections between the nodes. Hence, minimizing
$\|S\|_0$ corresponds to reducing the overall number of resistive components.


\subsection{Statement of contribution}

Regarding Problem~\ref{problem_id},
we will show that it is convenient to parameterize the set of solutions
as $T= P Q \sqrt{G}$, where $P$, $Q$ is the polar decomposition
of $T \sqrt{G}^{-1}$. In particular,
\begin{itemize}
\item Proposition~\ref{prop_D_G} shows that $P$ and $G$  are the solution
  of a problem with a reduced number of
  unknowns. Essentially, this result leverages the
  symmetry of $S$. Further, if either the second or the third
  condition in~(\ref{eqn_conditions}) is missing, $P$ and $G$ are the solution
  of a convex problem. We will consider this last case in more detail.
 We will mention that, in many cases, the solution $P$, $G$
  is unique up to a scaling factor.
  \item Proposition~\ref{prop_solut_Q} parameterizes all solutions
    of~$Q$ corresponding to each solution $P$, $G$.
  \end{itemize}

  These results can also be used to reduce the number of
  unknowns in Problems~\ref{problem_mer} and~\ref{problem_mimin}. Anyway, these
  last two problems are more difficult than
  Problem~\ref{problem_id}. To solve them, we will resort to
  general local search algorithms.

\subsection{Comparison with literature}

A problem similar in structure to Problem~\ref{problem_id}, but more general,
consists in solving the
following system with respect to unknown parameter vector $\theta$
\begin{equation}
  \label{eqn_conditions_gen}
  \left\{
  \begin{array}{ll}
T^{-1} \hat A T =  A(\theta) \\
    \hat C T= C(\theta) \\
T^{-1} \hat B =  B(\theta)\,,
  \end{array}
  \right.
\end{equation}
in which matrices $A$, $B$, $C$ depend on $\theta$.
Problem~\ref{problem_id} may be considered a special case
of~(\ref{eqn_conditions_gen}), in which $C$ and $B$ do not depend on
$\theta$ and the only constraint on $A$ is symmetry.

Problem~(\ref{eqn_conditions_gen}) has been extensively studied in recent literature.
A common approach consists in two
phases. First, one 
finds a black-box model (this is in general an easy one, for instance,
resorting to subspace-based methods). Second, one finds a coordinate
transformation $T$ and a parameter vector $\theta$ that
satisfy~(\ref{eqn_conditions_gen}). In the general case, this second
step is not an easy
task. In fact, even assuming, as commonly done, that $A$, $B$, $C$ are
a linear function of $\theta$, Problem~(\ref{eqn_conditions_gen}) is
bilinear and, thus, nonconvex.

This approach has been introduced
in~\cite{xie2002estimate} and studied in various subsequent works.
For instance,~\cite{parrilo2003initialization} studied the
problem of parameter
initialization. Works~\cite{6883189},~\cite{YU201854},~\cite{8844819},~\cite{8926483} present different
numerical approaches for the solution.

With respect to these general approaches, this work leverages the
special structure imposed by Problem~\eqref{eqn_RC_intro}, in
particular the symmetry of $S$ to obtain specific properties (see
Propositions~\ref{prop_D_G}  and~\ref{prop_solut_Q})
that do not apply to the more general
case~\eqref{eqn_conditions_gen}. As far as we know, the results
presented in these two propositions are new, perhaps also due to the
specificity of the problem discussed in this paper.

In literature, we can also find approaches for network topology
reconstruction not based on structured identification, such as~\cite{gonccalves2008, doddi2018,
  talukdar2020}.

{\bf Notation:} 
Matrix $A \in \Real^{n\times n}$ is diagonalizable
if there exist a nonsingular matrix $V$ and a diagonal
matrix $\Lambda$ such that $A V=V \Lambda$. The columns
of $V$ are the right eigenvectors of $A$.
Set $W=V^{-1}$, then we have that $W A= V^{-1} A= \Lambda
V^{-1}=\Lambda W$, which shows that  the rows of $W$ are left
eigenvectors of $A$. 
We also write that $(V,\Lambda,W)$ is a diagonalization of $A$. We denote the orthogonal group over
$\Real$ by
\[
  \begin{array}{ll}
O(n)=\{M \in \Real^{n \times n} \textrm{ such that } M
\textrm{ is invertible}  \textrm{ and } M^T M = I\},
\end{array}
    \]
that is the set of real orthonormal matrices of dimension $n$.
Given a subspace $V \subset \Real^n$, we will denote by $V^\perp$ its
orthogonal subspace.

\section{Discussion of Problem~\ref{problem_id}}

The following proposition presents a necessary condition for the
feasibility of Problem~\ref{problem_id}.

\begin{proposition}
\label{prop_two_cond}
Problem~\ref{problem_id} has a solution only if $\hat A$ is diagonalizable and has real eigenvalues.
\end{proposition}
\begin{proof}
   Assume that Problem~\ref{problem_id} has a solution.
      By the first of~(\ref{eqn_conditions}) it follows that
      $G^{1/2} T^{-1} \hat A T G^{-1/2}= G^{-1/2} S G^{-1/2}$. Note
      that this last matrix is symmetric, hence $\hat
      A$ is diagonalizable and has real eigenvalues, being similar to a symmetric matrix.
\end{proof}

Due to the previous proposition, we will make this assumption
throughout the paper.
\begin{assumption}
  \label{assum_diag}
$\hat
      A$ is diagonalizable and has real eigenvalues
\end{assumption}

\begin{remark}
By Proposition~\ref{prop_two_cond}, if Assumption~\ref{assum_diag}
does not hold for $\hat A$, then Problem~\ref{problem_id} does not
have a solution. This means that the identified system has not the
structure of an RC-network.
\end{remark}

We will parameterize the set of possible solutions $T$
of Problem~\ref{problem_id} as
\begin{equation}
  \label{eqn_par_of_T}
T= P Q \sqrt{G}\,,
\end{equation}
where $P$ is a symmetric and positive definite matrix and $Q \in O(n)$.
Note that $P Q$ corresponds to the left polar decomposition of $T
\sqrt{G}^{-1}$, which is unique, being $T \sqrt{G}^{-1}$ invertible.
In particular, $P$ corresponds to a scaling and $Q$ to a rotation or
reflection, further $P=\sqrt{T G^{-1} T^T}$.
In the following, we will show that
parameterization~\eqref{eqn_par_of_T} is convenient since couple $P$, $G$ can be found
separately from $Q$.
As a first step, the following proposition shows that the feasibility
of Problem~\ref{problem_id} is equivalent to the existence of a
solution of an equation independent of $Q$. The proof is presented in the Appendix.

\begin{proposition}
  \label{prop_M_E}
Problem~\ref{problem_id} has a solution $T, S, G$ if and only if there exists a
positive definite matrix $M$ such that
\begin{equation}
  \label{eqn_cond_eq}
  \left\{
\begin{array}{lll}
  \hat A M=M \hat{A}^T\\
  \hat C M \hat C^T=C G^{-1} C^T\\
  \hat B^T M^{-1} \hat B = B^T G^{-1} B\\
  \hat C \hat B= C G^{-1} B.
\end{array}
\right.
\end{equation}

Moreover, $M=P^2$, where $P$ is defined as in~\eqref{eqn_par_of_T}\,.
 \end{proposition}

 \vspace{10pt}
 
Note that, with respect to form~\eqref{eqn_par_of_T}, equation~\eqref{eqn_cond_eq} contains variables $M=P^2$, $G$,
but does not contain $Q$.
The structure of~\eqref{eqn_cond_eq} can be simplified by diagonalizing
$\hat A$. It particular, if
$(V,\Lambda,W)$ is a diagonalization of $\hat A$ (i.e., $\hat A=V
\Lambda W$), the following proposition
shows
that the first equation in~\eqref{eqn_cond_eq} can be substituted with
$M=V D V^T$, where $D$ is a matrix that commutes with
$\Lambda$ (i.e., $D \Lambda=\Lambda D$).  
\begin{proposition}
  \label{prop_diagon}
  Let  $(V,\Lambda,W)$ be a diagonalization of $\hat{A}$
 and let $M \in \Real^{n \times n}$, then 
the following statements are equivalent

i)  $\hat A M=M \hat{A}^T$

ii) there exists a matrix $D$, that commutes with $\Lambda$, such that $M=V D V^T$.

\end{proposition}   

\begin{proof}
i) $\Rightarrow$ ii) Substituting $\hat A=V \Lambda V^{-1}$ in i) we obtain 
$V \Lambda V^{-1} M =M V^{-T} \Lambda V^{T}$, which implies
$\Lambda V^{-1} M V^{-T} =V^{-1} M V^{-T} \Lambda$. Set $D=V^{-1} M V^{-T}$, then
$\Lambda D= D \Lambda$ and $M=V D V^T$.

ii) $\Rightarrow$ i) Let $D$ be any matrix such that $\Lambda D=D
\Lambda$ and set $M=V D V^T$. Then
$\hat A M=\hat A V D V^T= V \Lambda V^{-1} V D V^T= V \Lambda D V^T=
V D \Lambda V^T= V D V^T V^{-T} \Lambda V^T= M \hat A^T$.

\end{proof}





    

\begin{remark}
 \label{rem_D_comm_Lambda}
The requirement that $D$ commutes with $\Lambda$ limits the actual number of unknown
entries of $D$. In fact, setting
$\Lambda=\diag\{\lambda_1,\ldots,\lambda_n\}$, $D=(d_{ij})$ commutes with
$\Lambda$ if and only if
\[
  d_{ij}=0, \textrm{ for all } i,j \textrm{ such that } \lambda_i \neq \lambda_j\,.
\]
For instance, if all eigenvalues of $A$ are distinct, $D$ must be
diagonal. In the general case, $D$ has a block-diagonal structure.
\end{remark}

Combining the results of Propositions~\ref{prop_M_E}
and~\ref{prop_diagon}, we derive the following result.

\begin{proposition}
    \label{prop_D_G}
Problem~\ref{problem_id} has a solution $T, S, G$ if and only if there
exist
a symmetric matrix $D$ and a diagonal matrix $G$ such that

\begin{equation}
  \label{eqn_problem_nonl}
  \left\{
    \begin{array}{lll}
      D >0 \\
  \Lambda D= D \Lambda \\
   G > 0\\
  \hat C V D V^T \hat C^T=C G^{-1} C^T\\
  \hat B^T W^T D^{-1} W \hat B = B^T G^{-1} B\\
  \hat C \hat B= C G^{-1} B.
    \end{array}
 \right.
\end{equation}

Moreover, $P=\sqrt{V D V^T}$, where $P$ is defined in~\eqref{eqn_par_of_T}.
\end{proposition}

In Problem~\eqref{eqn_problem_nonl} the optimization variables are $D$
and $G$. This problem is nonconvex, since variable $D$ appears in it together with its inverse. 
\begin{remark}
  If the third condition is not present in (\ref{eqn_conditions}),
the third and fourth conditions disappear from~\eqref{eqn_cond_eq} and,
  setting $H=G^{-1}$, problem~\eqref{eqn_problem_nonl} reduces to a
  convex one:
\begin{equation}
  \label{eqn_prob_conv}
  \left\{
  \begin{array}{lll}
    D >0\\
  \Lambda D= D \Lambda \\
   H > 0\\
  \hat C V D V^T \hat C^T=C H C^T\,.
  \end{array}
  \right.
\end{equation}

The fact that the third condition is not present in
(\ref{eqn_conditions}) means that we do not impose any structural
requirement on matrix $B$. The solution of~\eqref{eqn_prob_conv}
is not unique. In fact, if $D$, $H$ is a solution
of~\eqref{eqn_problem_nonl}, any scaling $\alpha D$, $\alpha H$, with
$\alpha>0$ is still a solution.

In particular, if $\hat A$ has distinct eigenvalues, $D$ must be diagonal
and Problem~\eqref{eqn_prob_conv} reduces to finding positive diagonal
matrices $D$, $H$ such that
\begin{equation}
  \label{eqn_lin_case}
\hat C V D V^T \hat C^T=C H C^T.
\end{equation}
The set of all solutions of~\eqref{eqn_prob_conv}
corresponds
to a polyhedral cone and can be expressed as a conical combination of
a finite set of vertices (by Weyl-Minkowski theorem), that is, we can find vectors
$v_1,\ldots,v_l$ (called generators) such that the set of all solutions 
of~\eqref{eqn_prob_conv} is
\begin{equation}
  \label{eqn_sol_gen}
\{\alpha_1 v_1+\alpha_2 v_2+ \ldots +
  \alpha_l v_l, \quad \alpha_1,\alpha_2,\ldots, \alpha_l >0\}.
\end{equation}

These considerations also hold if the second
condition of (\ref{eqn_conditions}) is not present.
\end{remark}

\begin{remark}
  \label{rem_num_sol}
We present an intuitive discussion on the number of distinct solutions of~\eqref{eqn_prob_conv}.
Assuming $\hat A$ has distinct eigenvalues, so that $D$ is diagonal,
Problem~\eqref{eqn_prob_conv} reduces to finding positive diagonal
matrices $D$, $H$ such that~\eqref{eqn_lin_case} holds.
The solutions of~\eqref{eqn_lin_case} are represented by
vector $x=[\diag{D}, \diag{H}]$, that contains the elements on the
diagonal of the two matrices $D$ and $H$.

Note that, since the left and right-hand sides of this
expression are symmetric $p \times p$ matrices, this corresponds to a set of
$n_e=\frac{p(p+1)}{2}$ equations. We have $n_u=2 n$ unknown terms (the
diagonal elements of $D$ and $H$).
Hence, for a generic choice of problem data (i.e., matrices $\hat A$,
$C$, $\hat C$ are randomly selected) we have a solution consisting of
a unique ray (that is, unique up to scaling) if $n_e \geq
n_u-1$, that is $p \geq \frac{\sqrt{16n-7}-1}{2}$, where the $-1$ term
is due to the fact that a ray has dimension one.
However, if the problem data are not generic, we may have multiple
solutions even if this condition is satisfied. For instance, if $C$ is
a projection on the first $p$ components, then term $C H C^T$ does not
contain the last $n-p$ elements of the diagonal of $H$. Hence, these
are left undetermined and can be chosen as arbitrary positive
values. In this case, the number of remaining unknowns is $n_u=n+p$,
so that, if remaining parameters $\hat A$, $\hat C$ are generic, we
have only one solution for $D$ and for the first $p$ elements of the
diagonal of $H$ (up to a scaling factor) if $n_e \geq n_u-1$, that is
\begin{equation}
  \label{eq_cond_minim_m}
  p \geq \frac{\sqrt{8n-7}+1}{2}\,.
\end{equation}
These considerations intuitively justify the fact that, in generic cases, if $p$
is sufficiently high, equation~\eqref{eqn_prob_conv} has only one
solution (up to a scaling factor). Our numerical experiments confirm
this fact, however, we do not have a formal
proof.
\end{remark}

\begin{example}
  
  \label{example_1}

  Consider the RC circuit depicted in Figure~\ref{fig_circ_ex1}.
  \begin{figure}
    \label{fig_rc_intro}
    \begin{center}
\begin{tikzpicture}
\path (0,0) coordinate (ref_gnd);

\draw
  (ref_gnd) node  (cond3) {} to [C=\(C_3\)] (0,2)
  (2,0)  node  (cond1) {} to [C=\(C_4\)] (2,2)
  (4,0)  node  (cond4) {} to [C=\(C_1\)] (4,2)
  (6,0)  node  (cond2) {} to [C=\(C_2\)] (6,3)
  (0,3) to[R=\(R_1\)]  (4,3)
  (0,2) to[R=\(R_2\)]  (2,2)  
  (2,2) to[R=\(R_3\)]  (4,2) 
  (4,3) to[R=\(R_4\)]  (6,3) 
  (0,2) -- (0,3)
  (4,2) -- (4,3)
  (0,0) -- (6,0)

 (2,2) node[above] {D}
 (0,3) node[above] {C}
 (4,3) node[above] {A}
 (6,3) node[above] {B}
;

\fill[color=black] (ref_gnd)++(0,3)   circle[radius=0.08];
\fill[color=black] (ref_gnd)++(2,2) circle[radius=0.08];
\fill[color=black] (ref_gnd)++(4,3)   circle[radius=0.08];
\fill[color=black] (ref_gnd)++(6,3) circle[radius=0.08];
\end{tikzpicture}
\caption{Circuit used in Example~\ref{example_1}.}
\label{fig_circ_ex1}
\end{center}
\end{figure}
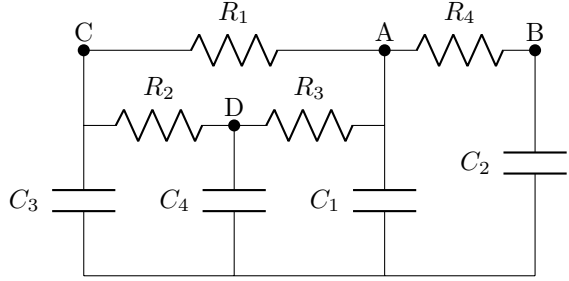
If  $x(t)\in \mathbb{R}^{4\times 1}$ represents the node potentials,
$C_1=C_2=C_3=C_4=1$, $R_1=R_4=R_2=1$, $R_3=2$,
the model of the system corresponds to (\ref{eqn_RC_intro}) with $G=I$ and
\begin{equation}
  \label{eqn_S_ex1}
  S=\left(\begin{array}{cccc} -4 & 1 & 1 & 2\\ 1 & -1 & 0 & 0\\ 1 &
                                                                       0 & -2 & 1\\ 2 & 0 & 1 & -3 \end{array}\right).
\end{equation}                                                                  

We assume that the system is autonomous (i.e., $B$ is not present) and
the outputs are the potentials of the first three nodes, that is
\[
C=\left(\begin{array}{cccc} 1 & 0 & 0 & 0\\ 0 & 1 & 0 & 0\\ 0 & 0 & 1
                                      & 0 \end{array}\right).
                                  \]

Suppose that we do not known matrices $G$ and $S$,
but we do know matrix $C$, since our output consists in the
potentials of the first three nodes. Assume also that, by using
standard identification techniques (for instance
subspace-based methods), we are able to identify a state-space model in
form
\[
  \begin{array}{ll}
    \dot z(t)= \hat A z(t)\\
    y(t)=\hat C z(t)\,,
    \end{array}
 \]
    with
{\small
    \[
\hat A=\left(\begin{array}{cccc} -10 & -4 & -23 & 5\\ 1 & -1 & 3 & -1\\ 3 &
                                                                       1
                                          & 7 & -2\\ 1 & -1 & 3 &
                                                                  -4 \end{array}\right),                                                                                               \hat C= \left(\begin{array}{cccc} 1 & 0 & 3 & -1\\ 0 & 1 & 0 & 0\\ 0 & -1 & 1 & 0 \end{array}\right).
\]}

Then, solving Problem~\ref{problem_id} consists in finding a
coordinate transformation $x=T z$ and matrices $G$ and $S$ such
that~\eqref{eqn_conditions} holds. In particular, matrix
$S$ is a key piece of information since it allows to reconstruct
the network structure.
Note that assumption~\ref{assum_diag} is satisfied.
We use parameterization~\eqref{eqn_par_of_T} and
apply Proposition~\ref{prop_D_G} to find matrices $P$ and $G$.
Since we do not have any requirement on input matrix $B$, we have to
solve convex problem~\eqref{eqn_prob_conv}.
Further, since $\hat{A}$ has distinct eigenvalues, by
Remark~\ref{rem_D_comm_Lambda}, $D$ must be diagonal, so that
Problem~\eqref{eqn_prob_conv} reduces to the following linear one, in
which
the variables are the diagonal matrices $D$ and $H$,
\begin{equation}
  \label{eqn_prob_conv_ex}
  \left\{
  \begin{array}{lll}
   D >0\\
   H > 0\\
  \hat C V D V^T \hat C^T=C H C^T\,.
  \end{array}
  \right.
\end{equation}

As previously noted, the solution to this problem is not unique,
since,
if $D$ and $H$ are a solution of~\eqref{eqn_prob_conv_ex} for any
$\alpha \in \Real$, with $\alpha >0$, also $\alpha D,
\alpha H$ is a solution of~\eqref{eqn_prob_conv_ex}.
The solution set has form~\eqref{eqn_sol_gen}, in particular there are
two generators,  so that the set of all solutions is given by
\begin{equation}
  \label{eqn_sol_ex_1}
x=\{v_1 \alpha_1,v_2 \alpha_2\},
\end{equation}
with $v_1=  \left(\begin{array}{cccccccc} 1.98 & 11.334 & 7.5 & 3.186
                    & 1 & 1& 1 & 0 \end{array}\right)$, \\ $v_2= \left(\begin{array}{cccccccc} 0 & 0 & 0 & 0 & 0 & 0 & 0 & 1 \end{array}\right)$.

All solutions are parameterized by positive parameters $\alpha_1$, $\alpha_2$. In particular, $\alpha_2$
is related to the fact that we cannot know the capacitance of the
unmeasured node. Moreover, $D$ depends only on
$\alpha_1$, so that the $P$ component of the solution $T$
in unique apart from an unknown scaling factor.

\vspace{10pt}

\end{example}




At this point, we compute the rotation component $Q$ of 
parameterization~\eqref{eqn_par_of_T}.
 
Let $P=\sqrt{M},G$ be solutions of~\eqref{eqn_cond_eq} and let $T$ be defined as
in~\eqref{eqn_par_of_T}. Then, substituting $T$ in the second and
third of~\eqref{eqn_conditions}, we obtain the following conditions
\begin{equation}
  \label{eqn_cond_Q}
  \begin{array}{ll}
\hat C P Q \sqrt{G}= C \\
\sqrt{G} Q^T P^{-1} \hat B =  G^{-1} B\,.
\end{array}
    \end{equation}
These conditions can be rewritten as $Q Z=  W$, where
\begin{equation}
    \label{eqn_S_Z}
W=\left(\begin{array}{ll}
               P \hat C^T&
               P^{-1} \hat B
              \end{array}\right),\,Z=\left(\begin{array}{ll}
               \sqrt{G}^{-1} C^T&
               \sqrt{G}^{-1}  B
              \end{array}\right)\,.
\end{equation}

Note that, by the second, the third and the fouth
of~\eqref{eqn_cond_eq}, $Z^TZ=W^TW$, so that
$Z$ and $W$ have the same rank.

In the following computations, it is convenient to assume that $Z$ and $W$
are full column-rank since, in this case, their left inverses $Z^+$
and $W^+$ are well-defined. 
If $Z$ and $W$ are not full column-rank, it is possible to reduce them
to full column-rank matrices by right multiplying them by a suitable
matrix $L$, as a consequence of the following simple algebraic property.
\begin{proposition}
Let $n,m,r$ be positive natural numbers, $Z,W \in \Real^{n \times m}$ with $Z^T Z=W^T W$, $Q \in \Real^{n
  \times n}$ and let $L \in
\Real^{m \times r}$ be such that $Z L$ is full column-rank and $\Img ZL
= \Img Z$, then the
following statements are equivalent:

i) $Q Z =W$

ii) $Q Z L = W L$.
\end{proposition}
\begin{proof}
  i) $\Rightarrow$ ii). Obvious. 

ii) $\Rightarrow$ i). By Proposition~\ref{prop_Gram}, being $Z^T Z = W^T W$, there exists
  $\hat Q \in O(n)$ such that $\hat Q Z = W$.  Since $\Img ZL=\Img Z$, there exists $M \in \Real^{r \times m}$ such
  that $Z=Z L M$. Then, $W L M=\hat Q Z L M = \hat Q Z = W$. Then i)
  is obtained by left-multipling ii) by $M$.
    \end{proof}

In the following, we will assume that $Z$ is full column-rank. In
fact, if this is not the case, it is sufficient to pick $L$ such that $Z
L$ is full column-rank and to redefine $Z=ZL$, $W=WL$.

            The following proposition shows that, if~\eqref{eqn_cond_eq} holds,
            there always exists an orthonormal matrix $Q$ that
            satisfies $Q Z=W$. We can distinguish two cases. First, in
            the trivial case in which
          $\textrm{rank } Z=n$ (that is, $Z$ is full row-rank)
          the solution is unique, as
          shown in the following Proposition.
          \begin{proposition}
        \label{prop_solut_Q_case1}
        Let $P, G$ be a solution of~\eqref{eqn_cond_eq},
        let  $W,Z$ defined as in~\eqref{eqn_S_Z}
        be such that $\textrm{rank } W=n$.
       Set $T=P Q \sqrt{G}$. Then, $T$ is a solution of Problem~\ref{problem_id} if and only if
              \begin{equation}
                \label{eqn_Q_case1}
                Q=W Z^{-1}\,.
\end{equation}
\end{proposition}

           \begin{proof}
             ($\Rightarrow$)
             By assumption $QZ=W$ has a solution, then, since
     $Z$ is full rank, it is invertible and $Q=WZ^{-1}$.
             ($\Leftarrow$)
             Assume that $Q$ is given by~(\ref{eqn_Q_case1}) and set $T= P
Q \sqrt{G}$, where $P$ and $G$ are a solution of~\eqref{eqn_cond_eq}.
Note that~\eqref{eqn_cond_eq} implies that $\hat A P^2=P^2 \hat A^T$, that
is $P^{-1} \hat A P= P \hat A^T P^{-1}$, so that $P^{-1}\hat A P$ is symmetric.
Then, also matrix
\[
\begin{gathered}
G T^{-1} \hat A T =  G \sqrt{G^{-1}} Q^T P^{-1}
\hat A P Q \sqrt{G} =\\
\sqrt{G} Q^T P^{-1}
\hat A P Q \sqrt{G} 
\end{gathered}
\] is symmetric,
proving the first of~(\ref{eqn_conditions}). Moreover,
$Q Z=W Z^{-1} Z=W$, which implies conditions~\eqref{eqn_cond_Q}.                    
\end{proof}

If $\textrm{rank }Z<n$, the solution $Q$ of $QZ=W$ is not unique,
            since, if  $Q$ satisfies $QZ =W$ and
            $\hat Q$ is any orthonormal matrix such
            that $\hat Q Z = Z$, then also $Q \hat Q Z =W$.
In fact, the following proposition shows that the set of possible
solutions $Q$ is parameterized by $O(n-\textrm{rank }Z)$.
      
      \begin{proposition}
        \label{prop_solut_Q}
        Let $P, G$ be a solution of~\eqref{eqn_cond_eq},
        let  $W,Z$ be defined as in~\eqref{eqn_S_Z} with $\textrm{rank }Z<n$
       and let $\bar W$, $\bar Z$ be orthonormal matrices
       such that $\Img \bar W=(\Img W)^\perp$, $\Img \bar Z=(\Img Z)^\perp$.
       Set $T=P Q \sqrt{G}$. Then, $T$ is a solution of Problem~\ref{problem_id} if and only if
              \begin{equation}
                \label{eqn_Q}
                Q=W Z^+ + \bar{W} \bar U \bar Z^T,
\end{equation}
                where
              $\bar U \in O(n-\textrm{rank }Z)$.
            \end{proposition}

            \begin{proof}
($\Rightarrow$) Let $P$, $G$ be a solution of~\eqref{eqn_cond_eq}
and set $T=P Q \sqrt{G}$. 
Then, $Q$ satisfies~\eqref{eqn_cond_Q} or, equivalently, $Q Z= W$.
Then, the thesis follows from Proposition~\ref{prop_orthog_par}.

($\Leftarrow$)
It is the same as the proof of the necessity
of Proposition~\eqref{prop_solut_Q_case1}, with the difference that, in this
case, $Q Z=W Z^+ Z+ \bar W \bar Q \bar Z^T Z=W$, which implies conditions~\eqref{eqn_cond_Q}.

\end{proof}

\addtocounter{example}{-1}
\begin{example} [continued]
  We consider again example~\ref{example_1}, and we choose a particular solution for
  $P$ and $G$ by setting $\alpha_1=\alpha_2=1$
  in~\eqref{eqn_sol_ex_1}. We made this choice in order to have $G=I$.
  We apply Proposition~\eqref{prop_solut_Q} to find the rotation component $Q$.
  In this case, $B$ is not present and $\dim (\Img W)^\perp=1$. For
  this reason $\bar{U} \in O(1)$, since $O(1)=\{-1,1\}$, there are two
  possible solutions for $Q$, given by
  \[
Q_1=W Z^+ + \bar{W} \bar Z^T,
Q_2=W Z^+ - \bar{W} \bar Z^T\,.
    \]

The corresponding transformation matrices $T_1, T_2$ are obtained
from~\eqref{eqn_par_of_T} and the symmetric part $S$
of~\eqref{eqn_conditions} by relation $S_i=G T_i^{-1} \hat A T_i$,
$i=1,2$, that is
{\small
\[
S_1=\left(\begin{array}{cccc} -4 & 1 & 1 & 2\\ 1 & -1 & 0 & 0\\ 1 & 0
                                     & -2 & 1\\ 2 & 0 & 1 &
                                                            -3 \end{array}\right),    S_2=   \left(\begin{array}{cccc} -4 & 1 & 1 & -2\\ 1 & -1 & 0 & 0\\ 1 & 0 & -2 & -1\\ -2 & 0 & -1 & -3 \end{array}\right).\]
}
These are all the solutions of Problem~\ref{problem_id} that correspond to
the chosen values for $P$ and $G$. Note that only $S_1$ is Metzler,
so that it is the only solution of Problem~\ref{problem_mer}, moreover
$S_1=S$, where $S$ is in~\eqref{eqn_S_ex1}.
Figures~\ref{fig:es1_vero} and~\ref{fig:es1_falso} represent the
graphs associated to matrices
$S_1$ and $S_2$. In these and in next graph figures, red nodes
denote unmeasured outputs.

\begin{figure}[!h]
    \centering
    \includegraphics[width=0.4\columnwidth]{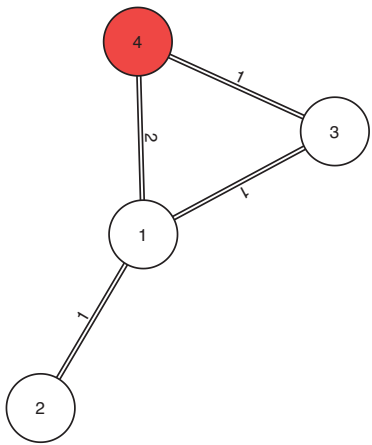}
    \caption{Graph representation of $S_1$}
    \label{fig:es1_vero}
  \end{figure}

  \begin{figure}[!h]
    \centering
    \includegraphics[width=0.4\columnwidth]{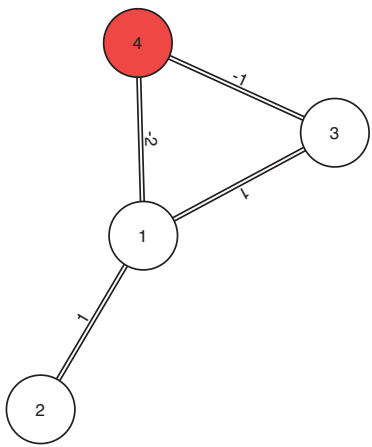}
    \caption{Graph representation of $S_2$}
    \label{fig:es1_falso}
\end{figure}

\end{example}

\section{Discussion of Problems~\ref{problem_mer} and~\ref{problem_mimin}}
\label{sect_problems_2_3}
Proposition~\ref{prop_solut_Q} shows that, in the general case,
Problem~\ref{problem_id}
has multiple solutions.
We introduced Problems~\ref{problem_mer} and~\ref{problem_mimin} in
order to find specific solutions that satisfy additional properties. Consider a solution of form~\eqref{eqn_par_of_T} and assume
  that $P$ and $G$ are fully known. Then, 
  Problem~\ref{problem_mer} consists in finding an orthonormal matrix
  $Q$ such that $T^{-1} \hat A T$ is Metzler, or equivalently, finding
  an orthonormal matrix
  $\bar U$ that satisfies the following equation.

  \begin{equation}
    \label{eqn_for_prob_2}
    \begin{array}{ll}
   \left(\sqrt{G}^{-1} Q^T P^{-1} \hat A P Q \sqrt{G}\right)_{i,j}\geq
      0,\,
   i \neq j\\
    Q=W Z^+ + \bar W \bar U \bar Z^T
\end{array}
    \end{equation}

    Note that Problem~\eqref{eqn_for_prob_2} is non-convex due to the
    orthonormality constraint on $\bar U$ (i.e., $\bar U^T \bar U=I$).
    Anyway, because of Proposition~\ref{prop_solut_Q}, the dimension
    of $\bar U$  may be small, so that, in some cases, solving~\eqref{eqn_for_prob_2} can still be a simple task.

    Problem~\ref{problem_mimin} adds the requirement of minimizing
    $\|S\|_0$, or, equivalently, minimizing $\|T^{-1} \hat A T\|_0$. Since the
    minimization of the zero-norm is a difficult task, as commonly
    done (see, for instance,~\cite{candes2008enhancing}), one can use
    the $1$-norm as a sparsity-promoting objective function, obtaining
    the following problem:
    \begin{equation}
      \label{eqn_for_prob_3}
 \min_{\bar U}  \left \|\sqrt{G}^{-1} Q^T P^{-1} \hat A P Q \sqrt{G}\right\|_1
\textrm{ such that~\eqref{eqn_for_prob_2} holds.}
\end{equation}

We can rewrite this problem more explicitly as
    \begin{equation}
      \label{eqn_for_prob_3_2}
       \begin{array}{ll}
 \underset{\bar U}{\min}  \left \|\sqrt{G}^{-1} Q^T P^{-1} \hat A P Q
   \sqrt{G}\right\|_1\\
\textrm{ such that}\\
   \left(\sqrt{G}^{-1} Q^T P^{-1} \hat A P Q \sqrt{G}\right)_{i,j}
         \geq 0,\,
   i \neq j\\
   Q=W Z^+ + \bar W \bar U \bar Z^T\\
   \bar U \bar U^T=I.
\end{array}
\end{equation}

\subsection{Overall algorithm for Problem~\ref{problem_mimin}}
\label{sec_alg}
Leveraging Proposition~\ref{prop_D_G}, we can formulate the following
algorithm for solving Problem~\ref{problem_mimin}.
Here the problem data are the identified model $\hat A$, $\hat B$,
$\hat C$ and the required input and output matrices $B$, $C$. The
final output is given by matrices $P$, $G$ and $\bar U$, that give transformation $T$ by~\eqref{eqn_par_of_T} and~\eqref{eqn_Q}.

\begin{itemize}
\item Solve Problem~\eqref{eqn_problem_nonl} in order to find a
  solution $P$, $G$. Note, that, by Remark~\ref{rem_num_sol}, in many
  cases, this solution is unique up to a scaling factor.
\item Solve Problem~\eqref{eqn_for_prob_3_2} using a nonlinear local
  search algorithm. In our tests we used different randomly generated initial conditions $\bar U_0$
  for $\bar U$ and selected the best solution.
  \end{itemize}

Some remarks are in order on the choice of the initial
  conditions $\bar U_0$ for $\bar U$. Note that $O(n)$ has two
  connected components given by
  $\{e^{S} A, S \in \Real^{n \times n}: S=-S^T, A \in \{I,M\}\}$, where $M$ is the diagonal matrix with
  all ones on the diagonal apart from a term $-1$ on the first element
  and $S$ is a skew-symmetric matrix. This comes from the facts the set
  of skew-symmetric matrices is the Lie algebra of $O(n)$ and that $I$ and $M$
  belong to separate connected components of $O(n)$.
Hence, we can generate a random initial guess $\bar U_0$ for $\bar U$
by setting $\bar U =e^{S} A$, where $S$ is a random skew-symmetric
matrix and $A$ is randomly chosen between $I$ and $M$.

\subsection{Case of data affected by noise}

Real input and output data are affected by noise. 
In this case, Problem~\ref{prop_D_G} may not have a feasible solution
and can be substituted with the following relaxed one.

\begin{equation}
  \label{eqn_problem_nonl_rel}
  \left\{
    \begin{array}{lll}
\min \|\hat C V D V^T \hat C^T-C G^{-1} C^T\|^2+
 \| \hat B^T W^T D^{-1} W \hat B -B^T G^{-1} B\|^2+
\| \hat C \hat B- C G^{-1} B\|^2\\
\textrm{subject to }\\
      D >0 \\
  \Lambda D= D \Lambda \\
   G > 0
    \end{array}
 \right.
\end{equation}
Note that if second or third conditions (the requirement on $B$ and
$C$) are not present in~\eqref{eqn_conditions},
Problem~\eqref{eqn_problem_nonl_rel} becomes a convex one. In fact if,
for instance, the third conditions is missing, by setting $H=G^{-1}$ the objective function
reduces to the convex one $\|\hat C V D V^T \hat C^T-C H C^T\|^2$.

\section{Examples}

In this section, we consider some examples of larger dimension.
In order to test the possibility of recovering the network connections,
we randomly generate some autonomous systems in
form~\eqref{eqn_RC_intro} (the ``true'' systems) with
the following procedure. Given a number of states $n$, we set
$G=I$ and $S=-\mathcal{I} K \mathcal{I}^T$, where $\mathcal{I}$ is the incidence
matrix
of a randomly generated graph of $n$ vertices and $K$ is a diagonal
matrix of randomly generated conductances (with integer values). Then,
we compute a random transformation matrix $\tilde T$ and set $\hat
A=\tilde T^{-1} A \tilde T$, $\hat C= C \tilde T$. We consider as output matrix $C$ the projection on the first $m$
components.
In all examples, $\hat A$ is diagonalizable and
condition~\eqref{eq_cond_minim_m} is
satisfied. Hence, $P$ and the first $m$ component along the
diagonal of $G$ have only one solution, up to a positive
scaling factor. The remaining elements of the diagonal $G$ are undetermined, since
they do not appear in Problem~\eqref{eqn_prob_conv_ex}. For
simplicity, we chose the scaling factor such that the reconstructed
$G$ is the identity.
We solved Problem~\ref{problem_mimin} with the algorithm presented in
Section~\ref{sec_alg} and computed the corresponding transformation
matrix $T$ and the reconstructed matrix $S$ as
$\hat S=G T^{-1} \hat A T$. Then, we compared matrix
$S$ of the true system with the reconstructed one $\hat S$, to check if we have been able
to correctly reconstruct the network connections. We considered the following two cases.

\subsection{Case 1: $n=10$, $m=8$}
In this example, we do not measure the potential of the last $2$ nodes,
that is, matrix $C$ in~\eqref{eqn_conditions} is the projection on the
first
$8$ nodes.
Generically (see Remark~\ref{rem_num_sol}), $P$ and $G$ are unique (up to a positive scaling), apart
from the last two components of the diagonal of $G$, that are
undetermined. By Proposition~\eqref{prop_solut_Q}, since $\textrm{ dim}
\textrm{ ker } C=2$, the component $Q$ of~\eqref{eqn_par_of_T} has
multiple solution, parameterized by $\bar U \in O(2)$. The algorithm
in Section~\ref{sec_alg} allows finding one among such solutions. Figure~\ref{fig:es2_1_vero} is the graph associated to $S$
while
Figure~\ref{fig:es2_1_falso} is the one associated to the
reconstructed $\hat S$. Note that the two graphs
are similar but different. That is, at the end of our procedure, we
found a reconstructed system of form~\eqref{eqn_RC_intro} which solves
Problem~\ref{problem_mer} (and, approximately,
Problem~\ref{problem_mimin}), but is different from the true
system. This is unavoidable since, by
Proposition~\eqref{prop_solut_Q}, there are multiple
systems that solve Problem~\ref{problem_id} and, in general, there may
be multiple solutions also of Problems~\ref{problem_mer} and~\ref{problem_mimin}.

\begin{figure}[!h]
    \centering
    \includegraphics[width=0.6\columnwidth]{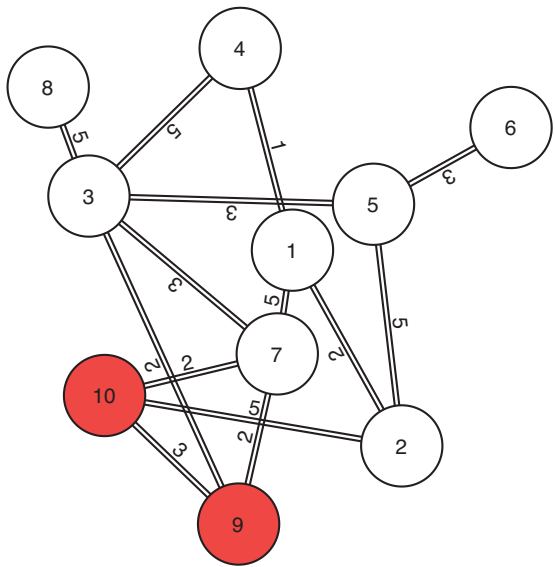}
    \caption{Case 1: true connections}
    \label{fig:es2_1_vero}
  \end{figure}

  \begin{figure}[!h]
    \centering
    \includegraphics[width=0.6\columnwidth]{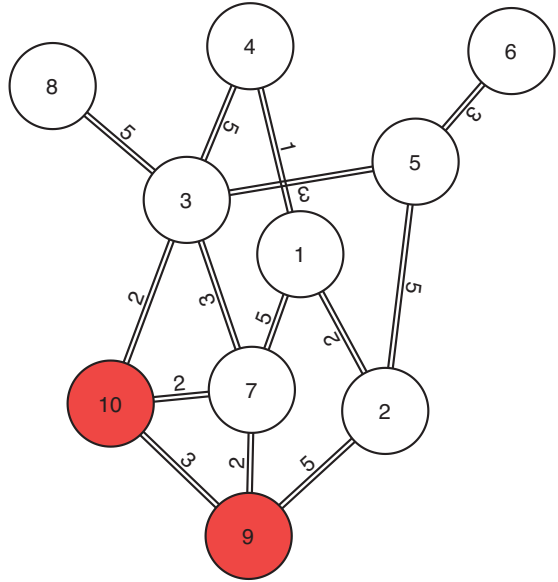}
    \caption{Case 1: reconstructed connections}
    \label{fig:es2_1_falso}
\end{figure}

\subsection{Case 2: $n=12$, $m=6$}
In this example, we measure the potential of $6$ nodes out of $12$.
Again, Figure~\ref{fig:es2_2_vero} refers to the true matrix $S$
while Figure~\ref{fig:es2_2_falso} refers to the reconstructed $\hat
S$.
In this case, $Q$ has multiple solutions parameterized by $\bar U \in O(6)$.
Again, the reconstructed matrix is different from the true
one, namely, at the end of our procedure, we found one of the multiple
solutions that solve Problem~\ref{problem_mer}.

\begin{figure}[!h]
    \centering
    \includegraphics[width=0.6\columnwidth]{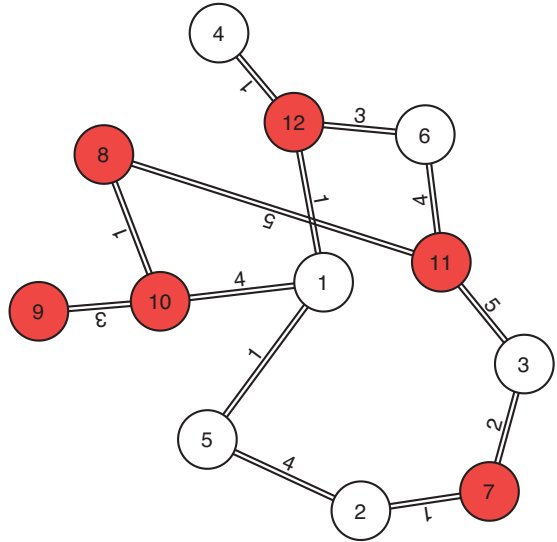}
    \caption{Case 2: true connections}
    \label{fig:es2_2_vero}
  \end{figure}

  \begin{figure}[!h]
    \centering
    \includegraphics[width=0.6\columnwidth]{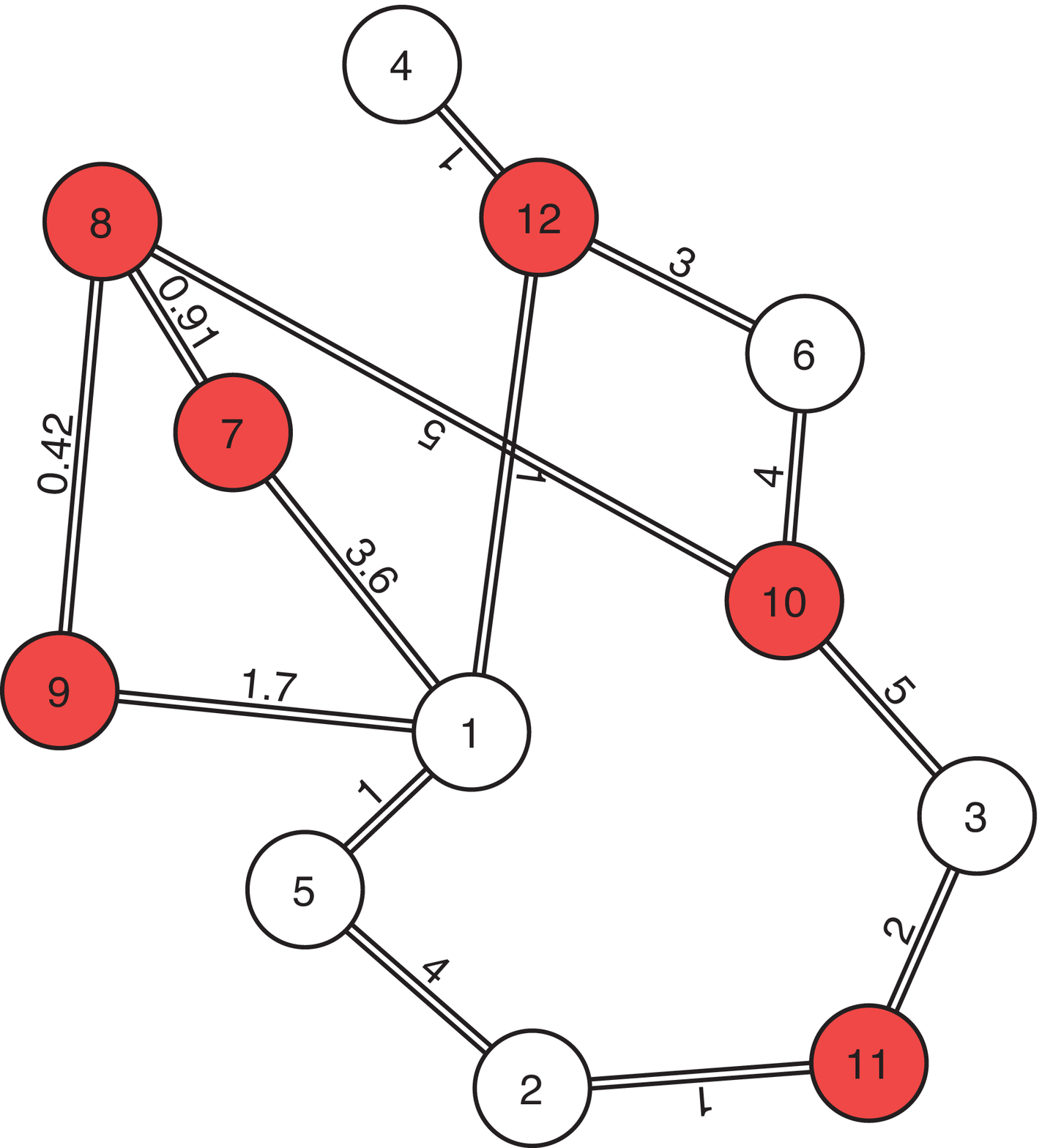}
    \caption{Case 2: reconstructed connections}
    \label{fig:es2_2_falso}
\end{figure}

\section{Conclusions}
Resistive-capacitive (RC) networks are used to model various systems in engineering, physics or
biology. We considered a structured identification task, characterized
the solution set and presented a possible algorithm for reconstructing
the network connections. 
\section*{Appendix}

\subsection*{Proof of Proposition~\ref{prop_M_E}}

      ($\Rightarrow$)
      Assume that~(\ref{eqn_conditions}) has a solution $T$, $S$, $G$, then
    $G T^{-1} \hat A T$ is symmetric, which implies that $G T^{-1}
    \hat A T = T^T
    \hat A^T T^{-T} G$, which, setting $M=T G^{-1} T^T$, implies the first of~(\ref{eqn_cond_eq}).
    The second of~(\ref{eqn_conditions}) implies $T^T\hat{C}^T =C^T$, thus  $\hat C T G^{-1} T^T \hat{C}^T= C G^{-1} C^T$, that
    is the second of~(\ref{eqn_cond_eq}).
    Similarly, the third of~(\ref{eqn_conditions}) implies $\hat B^T T^{-T} G T^{-1} \hat B= B^T G^{-1} B$, that
    is the third of~(\ref{eqn_cond_eq}). Finally, by the second and
    third of~\eqref{eqn_conditions} $C
     G^{-1} B = \hat C T T^{-1} \hat B= \hat C \hat B$.
      
      ($\Leftarrow$)
      Assume that~(\ref{eqn_cond_eq}) has a solution
      $M, G$. Let $UU^T=M$ be the Cholesky decomposition of $M$. The
      second, third and fourth conditions of~(\ref{eqn_cond_eq}) imply that
      \[
        \begin{gathered}
       \left(\begin{array}{ll}\hat C U,\\ \hat B^T U^{-T}\end{array}\right)
        \left(\begin{array}{ll}
U^T \hat C^T,\quad U^{-1} \hat B
              \end{array}
            \right) \\ =\left(\begin{array}{ll} C G^{-1/2}\\ B^T
                                G^{-1/2} \end{array} \right)
        \left(\begin{array}{ll}
                G^{-1/2} C^T,\quad
                G^{-1/2} B
              \end{array}
            \right)\,.
            \end{gathered}
        \]
        Then, by Proposition~\ref{prop_Gram}, there exists an orthonormal matrix $Q$ such that
         $\left(\begin{array}{ll}\hat C U\\ \hat B^T
           U^{-T}\end{array} \right) Q= \left(\begin{array}{ll} C
                                                G^{-1/2}\\ B^T G^{-1/2}\end{array}\right)$, that is
         \[
           \begin{array}{ll}
             \hat C U Q =C G^{-1/2}\\
             \hat B^T
           U^{-T} Q =B^T G^{-1/2}
             \end{array}
           \]
           and, setting $T=U Q G^{1/2}$, it follows that $\hat C T=C$
           and $T^{-1} \hat B=B$. Finally,   $\hat A M=M \hat{A}^T$
           implies that $\hat A=M \hat A^T M^{-1}=U U^T \hat A^T
           U^{-T} U^{-1}$ and  $ G T^{-1} \hat A T = G G^{-1/2} Q^T U^{-1} \hat A
           U Q G^{1/2}=G^{1/2}  Q^T U^{-1} U U^T \hat A^T U^{-T}
           U^{-1} U Q G^{1/2}$ $= G^{1/2}  Q^T U^T
           \hat A^T U^{-T} Q G^{-1/2} G = T^T \hat A^T T^{-T}
           G$. Hence $G T^{-1} \hat
           A T$ is symmetric, which proves the first of~(\ref{eqn_conditions}).

The following is a well-known property of Gram matrices (see for instance
Theorem~3.1 of~\cite{horn1996does}.

\begin{proposition}
  \label{prop_Gram}
  Let $A,B \in \Real^{n \times m}$ be such that
  $A^T A= B^T B$, then there exists $Q \in O(n)$ such that
  $A = Q B$.
\end{proposition}

The following proposition is a property of orthonormal transformations.

\begin{proposition}
  \label{prop_orthog_par}
  Let $A,B \in \Real^{n \times m}$ with $A^TA=B^T B$ and
  $\textrm{rank } A =m < n$,
  let $\bar A$, $\bar B$ be such that $\bar A^T \bar A=I$, $\bar
  B^T \bar B=I$, $\bar B^T B=0$, $\bar A^T A=0$,
  and let $Q \in \mathbb{R}^{n \times n}$.
  Let $\mathcal{R}=\{Q \in O(n): Q A = B\}$ and
  $\mathcal{S}=\{B A^++ \bar B U \bar A^T, U \in O(n-m)\}$,
  then $\mathcal{R}=\mathcal{S}$.
 \end{proposition} 

 \begin{proof}
(Proof that $\mathcal{R} \subset \mathcal{S}$.)

Let $Q \in \mathcal{R}$, then $Q A=B$ and $B^T Q \bar A=(Q A)^T Q \bar A=A^T \bar A=0$. Hence, $Q \bar A$ is orthogonal to
$B$ and the image of $Q \bar A$ belongs to the image of $\bar B$. This
implies that there exists a matrix $U$ such that $Q \bar A = \bar B U$.
Moreover, $\bar A^T Q^T Q \bar A=I=U^T \bar B^T \bar B U$ so that
$\bar B U \in O(n-m)$.
Then, $Q\vetdr{A} {\bar A}= \vetdr{B} {\bar B U}$.
Note that $\vetdc{A^+}{\bar A^T}\vetdr{A}{\bar A}= I$, so that
$\vetdc{A^+}{\bar A^T}=\vetdr{A}{\bar A}^{-1}$ and
$Q=\vetdr{B}{\bar BU}\vetdc{A^+}{\bar A^T}= B A^+ + \bar B U \bar A^T$.

(Proof that $\mathcal{S} \subset \mathcal{R}$.)

Let $U \in O(n-m)$, note that
   \[
\vetdc{B^T}{U^T \bar B^T} \vetdr{B}{\bar BU} 
=\matdxd{B^TB} {0} {0}{I}=\matdxd{A^TA}{ 0}{ 0}{ I}
=\vetdc{A^T}{\bar A^T} \vetdr{A}{\bar A}, 
   \]
then, by Proposition~\ref{prop_Gram}, there exists $Q \in O(n)$ such
that $\vetdr{B}{\bar BU} = Q  \vetdr{A}{\bar A}$, so that $Q A=B$.
Finally, since
$\vetdc{A^+}{\bar A^T}=\vetdr{A}{\bar A}^{-1}$, it follows that
$Q=\vetdr{B}{\bar BU}\vetdc{A^+}{\bar A^T}= B A^+ + \bar B U \bar A^T$.
\end{proof}

\bibliography{biblio}
\end{document}